\documentclass[10pt,leqno]{amsart}
\usepackage{graphicx}
\usepackage{indentfirst,csquotes}
\usepackage[indentafter]{titlesec}
\titleformat{name=\section}{}{\thetitle.}{0.8em}{\centering\scshape}
\titleformat{name=\subsection}[runin]{}{\thetitle.}{0.5em}{\bfseries}[.]
\titleformat{name=\subsubsection}[runin]{}{\thetitle.}{0.5em}{\itshape}[.]
\titleformat{name=\paragraph,numberless}[runin]{}{}{0em}{}[.]
\titlespacing{\paragraph}{0em}{0em}{0.5em}
\titleformat{name=\subparagraph,numberless}[runin]{}{}{0em}{}[.]
\titlespacing{\subparagraph}{0em}{0em}{0.5em}
\usepackage{lipsum}

\usepackage{amssymb,amsthm,amsmath}
\usepackage{xcolor,paralist,hyperref,titlesec,fancyhdr,etoolbox}
\usepackage[all]{xy}
\theoremstyle{definition}
\newtheorem{definition}{Definition}[]

\newtheorem{example}[definition]{Example}
\theoremstyle{plain}
\newtheorem{theorem}[definition]{Theorem}
\newtheorem{corollary}[definition]{Corollary}
\newtheorem{conjecture}[definition]{Conjecture}
\newtheorem{lemma}[definition]{Lemma}
\newtheorem{proposition}[definition]{Proposition}
\renewenvironment{proof}{\noindent\textsc{Proof.}\quad}{\qed}

\hypersetup{ colorlinks=true, linkcolor=black, filecolor=black, urlcolor=black }

\newcommand\Z{\mathop{}\!\mathbb{Z}}
\newcommand\Q{\mathop{}\!\mathbb{Q}}

\newcommand\ok{\mathop{}\!\mathfrak{o}}
\newcommand\OK{\mathop{}\!\mathfrak{O}}
\newcommand\Mat{\mathop{}\!\mathrm{Mat}}
\newcommand\Hom{\mathop{}\!\mathrm{Hom}}
\newcommand\Ext{\mathop{}\!\mathrm{Ext}}

\newcommand\GL{\mathop{}\!\mathrm{GL}}
\newcommand\Cl{\mathop{}\!\mathrm{Cl}}
\newcommand\mo{\mathop{}\!\mathrm{mod}}
\begin{document}
\title{Similarity of Matrices over Dedekind Rings} 
\author{Ziyang ZHU} 
\date{\today}
\address{School of Mathematical Sciences, Capital Normal University, Beijing 100048, China}
\email{zhuziyang@cnu.edu.cn}
\maketitle

\let\thefootnote\relax

\begin{abstract}
We extend Latimer and MacDuffee's theorem to a general commutative domain and apply this result to study similarity of matrices over integral rings of number fields. We also conjecture similarity over discrete valuation rings can be descent by a finite covering and verify this conjecture for $2\times2$ matrices and separable characteristic polynomials.
\end{abstract}

\bigskip

\section{Introduction}\label{s1}
It is well-known that two matrices over a field $k$ is similar over $k$ if and only if they are similar over a field extension $k'/k$. It is natural to ask if such property can be extended to certain rings, for example a Discrete Valuation Ring (DVR) $R$ with the fractional field $K$. Inspired by Serre and Grothendieck's conjecture \cite{pa} that the restriction map
\[H_{\mathrm{\acute{e}t}}^1(R,G)\to H_{\mathrm{\acute{e}t}}^1(K,G\times_R K)\]
 has a trivial kernel for a reductive group scheme $G$ over $R$, one can expect two matrices over $R$ are similar over $R$ if they are similar over $K$. However, such expectation is not true (see \S\ref{s4}). We wonder whether the descent property for similarity of matrices still holds. More precisely, we propose the following conjecture.

\begin{conjecture}\label{1}
Let $R$ be a DVR with the fractional field $K$ and $A, B$ be $n\times n$ matrices over $R$. Suppose that $L/K$ is a finite extension and $S$ is the integral closure of $R$ in $L$. If $A$ and $B$ are similar over $S$, then $A$ and $B$ are similar over $R$.
\end{conjecture}

It should be pointed out that such descent property does not hold over $\mathbb Z$. For example, consider two matrices
\[A:=\left(\begin{array}{cc}0&1\\-6&0\end{array}\right)\text{ and }B:=\left(\begin{array}{cc}0&2\\-3&0\end{array}\right)\]
in $\Mat_2(\Z)$. If there exists a matrix $U:=\left(\begin{array}{cc}x&y\\z&w\end{array}\right)\in\GL_2(\Z)$ such that $UA=BU$, then
\[\det(U)=2w^2+3y^2=\pm1\]
has a solution in $\Z$, which is absurd. But the equation has a solution $w=\sqrt{-1},y=1$ in the Gaussian ring $\Z[\sqrt{-1}]$, which means $A$ and $B$ are similar over $\Z[\sqrt{-1}]$.

But \cite[Theorem 7]{gu} proved that, for a ring $\mathcal{O}_K$ of algebraic integers, $A,B\in\Mat_n(\mathcal{O}_K)$ are similar over all the local rings of $\mathcal{O}_K$, if and only if $A,B$ are similar over some finite integral extension of $\mathcal{O}_K$. This partially proves Conjecture \ref{1}, since DVRs are special Dedekind domains and \cite{rz} claimed '$\mathcal{O}_K$' can be replaced by a general Dedekind domain, if certain cohomological condition is satisfied. However, \cite[Theorem 7]{gu} did not mention the condition of cohomology. Indeed, this condition is not always satisfied in general cases (for instance, see Example \ref{7}).

On the other hand, in \cite{lm}, Latimer and MacDuffee provided an important tool to study the similarity of matrices over $\mathbb Z$. As pointed in \cite[Lemma 6.4]{wx}, this result holds for any Principal Ideal Domain (PID). Previous works by Sarkisjan \cite{sa} and Grunewald \cite{gr} showed the decidability of integral similarity computationally. Some specific methods are applied for the $2\times2$ case by Behn and Van der Merwe \cite{bv} (for a dynamical application, see also \cite{elmv}), and the $3\times3$ case by Appelgate and Onishi \cite{ao}. However, the higher dimensional situation is challenging, there are only some scattered results. For instance, \cite{pss} considered matrices over local rings of small length. On the computational front, Magma handles finite-order and $2\times2$ matrices efficiently \cite{ops}. Beyond that, David Husert's thesis \cite{hu} presented a more general algorithm.

In this paper, we first extend the result of Latimer and MacDuffee to a general integral domain (\cite{eg} provided a broader and coarser extension):

\begin{theorem}\label{2}
Let $R$ be an integral domain with fractional field $K$ and $f$ be a monic polynomial over $R$ with $(f,f')=1$, where $f'$ is the derivative of $f$. If $\GL_n(R)$ acts on the set $\Mat_n(R)$ of $n\times n$ matrices over $R$ by conjugation, then there is a one-to-one correspondence
\begin{multline*}
\big\{A\in\Mat_n(R):\det(xI_n-A)=f(x)\big\}\big/\GL_n(R)\\
\longleftrightarrow\big\{J\text{ ideal of }R[x]/(f):J\text{ a free $R$-module of rank $n$}\big\}\big/\sim,
\end{multline*}
where $J_1\sim J_2$ for two ideals $J_1$ and $J_2$ in $R[x]/(f)$ means that there are $\alpha_1$ and $\alpha_2$ in $R[x]/(f)$ which are not zero divisors such that $\alpha_1 J_1= \alpha_2 J_2$.
\end{theorem}

Unlike \cite{pi} and \cite{aopv}, we apply such correspondence to classifying the similarity classes of $2\times2$ matrices over a DVR. As an application of \cite{rz} and \cite{wx}, we show that:

\begin{theorem}\label{3}
Conjecture \ref{1} is true if
\begin{itemize}
\item the characteristic polynomial is separable.

\item $n=2$.
\end{itemize}
\end{theorem}

One can also apply Theorem \ref{2} to study the similarity of matrices over the ring of integers of a number field (see \S\ref{s5}). We introduce an algorithm that fully classifies the similarity classes of $n\times n$ matrices within this context and provide several examples to highlight the differences between this case and the scenario involving PIDs.

The paper is organized as follows. In \S\ref{s2}, we prove the first part of Theorem \ref{3} by completing the unnoticed parts in \cite{gu}. In \S\ref{s3}, we prove Theorem \ref{2}, which establishes a one-to-one correspondence between ideal classes and similarity classes of matrices. In \S\ref{s4}, as an application of Theorem \ref{2}, we classify the similarity classes of $2\times2$ matrices over a DVR and apply this classification to prove the remaining part of Theorem \ref{3}. In \S\ref{s5}, we study the similarity classes over the ring of algebraic integers.
\subsection*{Acknowledgments}
The author expresses his gratitude to Mingqiang FENG and Fei XU for their valuable discussions, they made significant contributions to \S\ref{s2}-\S\ref{s4}. Additionally, I would like to thank Zi'ang XU for providing me with valuable references.

\section{Cohomology is Torsion in Semi-simple Case}\label{s2}

Let $(R,\pi)$ be a DVR with fractional field $K$. Let $L/K$ be a finite field extension, and $S$ be the integral closure of $R$ in $L$. For a polynomial $f\in R[x]$, write $\ok=R[x]/(f)$ and $\OK=S[x]/(f)$. For a matrix $A\in\Mat_n(R)$, it induces an $\ok$-module $M_A:=R^{\oplus n}$ with $\ok$-action
\[(g(x)\mo~f)(r_1,\cdots,r_n):=(r_1,\cdots,r_n)g(A)^T.\]

Given two matrices $A,B\in\Mat_n(R)$ with the same separable characteristic polynomial $f$, we have:

\begin{lemma}[\cite{rz}]\label{4}
If $\Ext_{\ok}^1(M_A,M_B)$ is a torsion $R$-module, then $M_A\otimes_{\ok}\OK\cong M_B\otimes_{\ok}\OK$ (as $\OK$-modules) if and only if $A$ and $B$ are similar over $R$.
\end{lemma}

To facilitate our subsequent comments, it is necessary to reorganize the proof.

\begin{proof}
We just need to prove the 'only if' part. Let $t$ be any positive integer, write $\overline{\ok}:=\ok/\pi^t\ok$ and $\overline{\OK}:=\OK/\pi^t\OK$, then $\overline{M_A}:=M_A/\pi^tM_A$ is an $\overline{\ok}$-module and
\[\overline{M_A\otimes_{\ok}\OK}:=M_A\otimes_{\ok}\OK/\pi^t(M_A\otimes_{\ok}\OK)=\overline{M_A}\otimes_{\overline{\ok}}\overline{\OK}\]
is an $\overline{\OK}$-module since $\OK$ is a flat $\ok$-module. Suppose we have an $\OK$-module isomorphism $M_A\otimes_{\ok}\OK\cong M_B\otimes_{\ok}\OK$, then
\[\overline{M_A\otimes_{\ok}\OK}\cong\overline{M_B\otimes_{\ok}\OK}\]
as $\overline{\OK}$-modules. But as $\overline{\ok}$-modules,
\[\overline{M_A\otimes_{\ok}\OK}\cong\overline{M_A}^{\oplus[L:K]}\quad\text{and}\quad\overline{M_B\otimes_{\ok}\OK}\cong\overline{M_B}^{\oplus[L:K]}.\] Note that the ring $\overline{\ok}$ is Artinian, so $\overline{M_A}\cong\overline{M_B}$ as $\overline{\ok}$-modules by the Krull-Schmidt theorem.

Now we show, there exists suitable $t$, such that $\overline{M_A}\cong\overline{M_B}$ implies $M_A\cong M_B$ (as $\ok$-modules). As $R$-modules, the free modules $M_A$ and $M_B$ are always isomorphic, so the isomorphism of $\overline{\ok}$-modules $\overline{M_A}\cong\overline{M_B}$ must come from some $R$-module isomorphism $\eta:M_A\to M_B$. We see this $\eta$ satisfies
\[(\eta(am)-a\eta(m))\in\pi^tM_B,\quad\text{for all }a\in\ok,m\in M_A.\]
Recall that the adjunction pairing $(\otimes,\Hom)$ provides a natural isomorphism
\[\Gamma:\Hom_R(M_A,M_B)=\Hom_R(M_A,\Hom_{\ok}(\ok,M_B))\overset{\sim}{\longrightarrow}\Hom_{\ok}(M_A\otimes_R\ok,M_B),\]
given by $\varphi\longmapsto\big[\Gamma(\varphi):m\otimes a\mapsto a\varphi(m)\big]$, where we view $\ok$ as an $(R,\ok)$-bimodule. Select the following projective resolution of $\ok$-module $M_A$:
\[\cdots\longrightarrow K_2\otimes_R\ok\overset{\alpha_1}{\longrightarrow}K_1\otimes_R\ok\overset{\alpha_0}{\longrightarrow}M_A\otimes_R\ok\longrightarrow M_A\longrightarrow0,\]
where
\[0\to K_2\to K_1\otimes_R\ok\to K_1\to0\quad\text{and}\quad0\to K_1\to M_A\otimes_R\ok\to M_A\to0.\]
Applying the functor $\Hom_{\ok}(\square,M_B)$ to the above long exact sequence, we obtain a complex
\[\xymatrix@C=0.5cm{\Hom_{\ok}(M_A,M_B)\ar@{^(->}[r]^{i~}&\Hom_{\ok}(M_A\otimes_R\ok,M_B)\ar[r]^{\alpha_0^*}&\Hom_{\ok}(K_1\otimes_R\ok,M_B)\ar[r]^{\alpha_1^*}&\Hom_{\ok}(K_2\otimes_R\ok,M_B)}\]
Consider $\Gamma(\eta)\in\Hom_{\ok}(M_A\otimes_R\ok,M_B)$, it is easy to verify $\alpha_0^*\Gamma(\eta)\equiv0$ as a map $K_1\otimes_R\ok\to M_B/\pi^tM_B$, hence $\alpha_0^*\Gamma(\eta)=\pi^tg_{\eta}$ for some $g_{\eta}\in\Hom_{\ok}(K_1\otimes_R\ok,M_B)$, and $0=\alpha_1^*\alpha_0^*\Gamma(\eta)=\pi^t\alpha_1^*(g_{\eta})$ implies $\alpha_1^*(g_{\eta})=0$, which means
\[g_{\eta}\in\Ext_{\ok}^1(M_A,M_B).\]
Since $\Ext_{\ok}^1(M_A,M_B)$ is a torsion $R$-module, there exists an integer $s$ such that $\pi^sg_{\eta}=\alpha_0^*\Gamma(h_{\eta})$ for some $h_{\eta}\in\Hom_R(M_A,M_B)$. Thus, $\alpha_0^*\Gamma(\eta-\pi h_{\eta})=\pi^tg_{\eta}-\pi^{s+1}g_{\eta}=0$ if we choose $t=s+1$, so there exists
\[F_{\eta}\in\Hom_{\ok}(M_A,M_B)\]
such that $i(F_{\eta})=\Gamma(\eta-\pi h_{\eta})$. But $\eta-\pi h_{\eta}\in\Hom_R(M_A,M_B)$ is an isomorphism because $\det(\eta-\pi h_{\eta})\equiv\det(\eta)(\mo~\pi)$ and $\det(\eta)\in R^{\times}$, so $F_{\eta}$ is an $\ok$-isomorphism, which implies $A$ and $B$ are similar over $R$.
\end{proof}

\begin{lemma}\label{5}
If $f$ is separable, then $\Ext_{\ok}^1(M_A,M_B)$ is a torsion $R$-module.
\end{lemma}
\begin{proof}
Since the homomorphism $\ok=R[x]/(f)\to K[x]/(f)$ is flat, we have
\[\Ext_{\ok}^1(M_A,M_B)\otimes_RK\cong\Ext_{\ok}^1(M_A,M_B)\otimes_{\ok}K[x]/(f)\cong\Ext_{K[x]/(f)}^1(M_A\otimes_RK,M_B\otimes_RK).\]
This is vanish, because $K[x]/(f)$ is a semi-simple ring.
\end{proof}

So the condition in Lemma \ref{4} is always holds if $f$ is separable. At this point, we can partially prove Theorem \ref{3}:

\begin{theorem}\label{6}
Conjecture \ref{1} is true if the characteristic polynomial is separable.
\end{theorem}
\begin{proof}
Given two matrices in $\Mat_n(R)$ with the same separable characteristic polynomial, we hope they are similar over $S$ implies they are similar over $R$. But this is a direct conclusion of Lemma \ref{4} and Lemma \ref{5}.
\end{proof}

However, if $f$ is inseparable, the extension module $\Ext_{\ok}^1(M_A,M_B)$ is not necessarily a torsion $R$-module (the condition in \cite{rz} is stronger than ours, it may certainly also fail).

\begin{example}\label{7}
Let $A=\left(\begin{array}{cc}0&1\\0&0\end{array}\right)$ over some DVR $R$ with characteristic polynomial $f(x)=x^2$, then $\ok=R[x]/(x^2)$. One can show as an $\ok$-module, $M_A\cong(x)/(x^2)$ is not projective. But
\[\Ext_{\ok}^1(M_A,M_A)\otimes_RK\cong\Ext_{K[x]/(x^2)}^1(M_A\otimes_RK,M_A\otimes_RK)\]
has a free rank generated by the short exact sequence
\[0\to(x)/(x^2)\to K[x]/(x^2)\to(x)/(x^2)\to0.\]
\end{example}

Therefore, merely using the results above cannot handle the inseparable situation, the conjecture remains an open problem. In section \S\ref{s4}, as a computable case, we focus on $2\times2$ matrices and prove the corresponding conjecture even if $f$ is inseparable.

\section{Similarity Classes and Ideal Classes}\label{s3}

In \cite{lm} (see also \cite[Theorem III.13]{new}), the similarity of $n\times n$ matrices over $\Z$ with a given characteristic polynomial without multiple roots has been studied. As pointed out in \cite[Lemma 6.4]{wx}, such a result is true over any PID.

\begin{proposition}\label{8}
Let $R$ be a PID with the fractional field $K$. Let $f\in R[x]$ be a monic irreducible polynomial and $\theta$ be a root of $f$ (this means that $\theta$ is the class of $x$ in $R[x]/(f)$), then there exists a bijection
\[\{A:A\text{ is a similarity class in }\Mat_n(R)\text{ with }f(A)=0\}\longrightarrow\{\text{ideal classes in }R[\theta]\},\]
given by $A\longmapsto R\langle X\rangle$, where the column vector $X\in(R[\theta])^n$ satisfies $AX=\theta X$.
\end{proposition}

Given an ideal class $R\langle X\rangle=R\langle x_1,\cdots,x_n\rangle$, the map in Proposition \ref{8} is invertible because $(\theta x_1,\cdots,\theta x_n)\in(R\langle X\rangle)^n$, and there must exist a matrix $A\in\Mat_n(R)$ such that $A(x_1,\cdots,x_n)^T=(\theta x_1,\cdots,\theta x_n)^T$.

In this section, we explain how to extend this result to an integral domain. The following lemma is well-known. For completeness, we provide a proof as well.

\begin{lemma}\label{9}
Let $R$ be an integral domain with the fractional field $K$. If $f\in R[x]$ is a monic polynomial, then the natural map
\[R[x]/(f)\hookrightarrow K[x]/(f)\]
is injective. Moreover, assuming that $(f,f')=1$ in $K[x]$ and $\alpha \in R[x]/(f)$, then $\alpha$ is not a zero divisor of $R[x]/(f)$ if and only if $\alpha$ is invertible in $K[x]/(f)$.
\end{lemma}
\begin{proof}
Let $h\in R[x]$ satisfying $h\in(f)$ over $K[x]$. Then there is $g\in K[x]$ such that $h=fg$. Since $f$ is a monic polynomial over $R$, one can write
\[f(x)=\sum_{i=0}^na_ix^i\quad\text{and}\quad g(x)=\sum_{i=0}^mb_ix^i,\]
where $a_n=1$, $a_{n-1},\cdots,a_0\in R$ and $b_m,\cdots,b_0\in K$. By comparing the coefficients of $x^l$ in $h=fg$, one obtains
\[\sum_{i=0}^la_ib_{l-i}\in R,\quad\text{for }0\leq l\leq m+n,\]
where $a_i=b_j=0$ for $i>n$ and $j>m$.

When $l=m+n$, one obtains $a_nb_m=b_m\in R$. Suppose $b_m,\cdots,b_k\in R$. Take $l=n+k-1$, one obtains
\[a_nb_{k-1}+\sum_{j=1}^na_{n-j}b_{j+k-1}=b_{k-1}+\sum_{j=1}^na_{n-j}b_{j+k-1}\in R.\]
This implies that $b_{k-1}\in R$. By induction, one concludes $g\in R[x]$. Therefore the natural map is injective.

For any non-zero element $\beta\in K[x]/(f)$, there is a non-zero element $\xi\in R$ such that $\xi\beta\in R[x]/(f)$ and $\xi\beta\neq0$ in $R[x]/(f)$. If $\alpha$ is not a zero divisor of $R[x]/(f)$, then $\alpha$ is not a zero divisor of $K[x]/(f)$ by the natural injection. Since $K[x]/(f)$ is a direct sum of fields by $(f,f')=1$, one concludes that $\alpha$ is invertible in $K[x]/(f)$. Conversely, if $\alpha$ is invertible in $K[x]/(f)$, then any non-zero element $\xi\in R[x]/(f)$ with $\alpha\xi=0$ implies that $\xi=0$ by the natural injection. Namely, $\alpha$ is not a zero divisor of $R[x]/(f)$.
\end{proof}

\begin{lemma}\label{10}
Let $R$ be an integral domain and $f\in R[x]$ be a polynomial. If $M$ is an $R$-module in $R[x]/(f)$ and $\alpha$ is not a zero divisor of $R[x]/(f)$, then the map
\[M\longrightarrow\alpha M,\quad x\longmapsto\alpha x\]
is an isomorphism of $R$-modules. In particular, $M$ is a free $R$-module if and only if $\alpha M$ is a free $R$-module.
\end{lemma}
\begin{proof}
Clearly, this map is an $R$-module homomorphism. Its kernel is $\{x\in M:\alpha x=0\}=0$, since $\alpha$ is not a zero divisor.
\end{proof}

\begin{definition}\label{11}
Let $f$ be a polynomial over an integral domain $R$. For two ideals $J_1$ and $J_2$ of $R[x]/(f)$, we say that $J_1\sim J_2$ if there are $\alpha_1$ and $\alpha_2$ in $R[x]/(f)$ which are not zero divisors such that $\alpha_1J_1=\alpha_2J_2$.
\end{definition}

It is clear that '$\sim$' is an equivalent relation among the set of all ideals of $R[x]/(f)$. If the degree of $f$ is $n$, then the subset
\[\{J\text{ ideal of }R[x]/(f):J\text{ is a free $R$-module of rank $n$}\}\]
induces the subset of equivalent classes of all ideals of $R[x]/(f)$ by Lemma \ref{10}.

Let us prove Theorem \ref{2} now.

\begin{proof}
Let $K$ be the fractional field of $R$ and $f(x)=x^n+a_1x^{n-1}+\cdots+a_n$ and fix
\[A_0=\left(\begin{array}{ccccc}0&0&\cdots&0&-a_n\\1&0&\cdots&0&-a_{n-1}\\\vdots&\vdots&\ddots&\vdots&\vdots\\0&0&\cdots&1&-a_1\end{array}\right)\in\Mat_n(R).\]

For any matrix $A\in\Mat_n(R)$ with the characteristic polynomial $f$, there is $g\in\GL_n(K)$ such that $A_0=gAg^{-1}$ by $(f,f')=1$. Choose $a\in R\setminus \{0\}$ such that $a\cdot g\in\Mat_n(R)$. Let
\[(u_1,u_2,\cdots,u_n)=(1,\overline{x},\cdots,\overline{x}^{n-1})\cdot a\cdot g\text{ and }J=Ru_1+\cdots+Ru_n.\]
Since \[\overline{x}\cdot(u_1,\cdots,u_n)=(1,\overline{x},\cdots,\overline{x}^{n-1})A_0\cdot a\cdot g=(u_1,\cdots,u_n)A,\quad\quad(\ast)\]
one obtains that $J$ is an ideal of $R[x]/(f)$. Moreover, since $a\cdot g\in\GL_n(K)$, one obtains that $J$ is a free $R$-module with a basis $\{u_1,\cdots,u_n\}$.

Let $R[A]$ and $K[A]$ be $R$-algebra and $K$-algebra generated by $A$ respectively. Then one has the following commutative diagram
\[\xymatrix{R[x]/(f)\ar[d]_{\iota}\ar[r]^{\phi}&R[A]\ar[d]^{\kappa}\\K[x]/(f)\ar[r]^{\psi}&K[A]}\]
by sending $x$ to $A$. Since $(f,f')=1$, the map $\psi$ is an isomorphism. By Lemma \ref{9}, the map $\iota$ is injective. This implies that $\phi$ is an isomorphism and $\kappa$ is injective.

Suppose $h\in\GL_n(K)$ satisfies $A_0=hAh^{-1}$. Then $g^{-1}hA=Ag^{-1}h$. Since $(f,f')=1$, one obtains that $g^{-1}h\in K[A]$. Choose $b\in R\setminus\{0\}$ such that $b\cdot h\in\Mat_n(R)$ and set
\[(v_1,\cdots,v_n)=(1,\overline{x},\cdots,\overline{x}^{n-1})\cdot b\cdot h\text{ and }J_1=Rv_1+\cdots Rv_n.\]
Then
\[(u_1,\cdots,u_n)\cdot(g^{-1}h)=(v_1,\cdots,v_n)\cdot(ab^{-1}).\]
By $(\ast)$, one has
\[\psi^{-1}(g^{-1}h)\cdot(u_1,\cdots,u_n)=(u_1,\cdots,u_n)\cdot(g^{-1}h).\]
Let $c\in R\setminus\{0\}$ such that $c\cdot\psi^{-1}(g^{-1}h)\in R[x]/(f)$. Since $g^{-1}h\in\GL_n(K)$, one concludes that $bc\cdot\psi^{-1}(g^{-1}h)$ and $ac$ are invertible in $K[x]/(f)$. By Lemma \ref{9}, both $bc\cdot\psi^{-1}(g^{-1}h)$ and $ac$ are not zero divisors of $R[x]/(f)$ and
\[bc\cdot\psi^{-1}(g^{-1}h)\cdot J=ac\cdot J_1.\]
Therefore $J\sim J_1$. This implies that the map $A\mapsto J$ is well-defined.

Conversely, for any ideal $J$ of $R[x]/(f)$ which is a free $R$-module with rank $n$, one can choose a $R$-basis $(u_1,\cdots,u_n)$ of $J$ and obtain a matrix $A\in\Mat_n(R)$ by
\[\overline{x}\cdot(u_1,\cdots,u_n)=(u_1,\cdots,u_n)A.\]
If $J_1$ is an ideal of $R[x]/(f)$ such that $J_1\sim J$, then $J_1$ is also a free $R$-module of rank $n$ by Lemma \ref{10}. Let $\{v_1,\cdots,v_n\}$ be a basis of $J_1$ over $R$ and
\[\overline{x}\cdot(v_1,\cdots,v_n)=(v_1,\cdots,v_n)B.\]
Since $J\sim J_1$, there are not zero divisors $\xi,\eta\in R[x]/(f)$ such that $\xi J=\eta J_1$. This implies that there is $T\in\GL_n(R)$ such that
\[(\xi u_1,\cdots,\xi u_n)=(\eta v_1,\cdots,\eta v_n)T.\]
Thus $B=TAT^{-1}$. The map $J \mapsto A$ is well-defined and gives the opposite direction map.
\end{proof}

\section{Over Discrete Valuation Rings}\label{s4}
In this section, let $(R,\pi)$ be a DVR, with fractional field $K$ and residue field $\kappa:=R/(\pi)$. The valuation of $R$ is denoted by $v:K^{\times}\to\Z$.

In order to study concrete cases, we focus on $2\times2$ matrices. Let $A\in\Mat_2(R)$ be a matrix with characteristic polynomial $f(x)=x^2-ax-b\in R[x]$. We will discuss the similarity classes in $\Mat_2(R)$ in the following several cases.

If $f$ is reducible (see \S\ref{s4.1}), \cite[Theorem III.12]{new} (which initially established for $\Z$, however, it is straightforward to verify this also hold for PIDs) almost directly resolves the problem:
\begin{proposition}\label{12}
Let $R$ be a PID with the fractional field $K$. Suppose $A\in\Mat_n(R)$, then $A$ lies in the similarity class represented by the following matrix
\[\left(\begin{array}{cccc}A_{11}&A_{12}&\cdots&A_{1r}\\
 0&A_{22}&\cdots&A_{2r}\\
 \vdots&\vdots&\ddots&\vdots\\
 0&0&\cdots&A_{rr}\end{array}\right)\]
where $A_{ii}$ are matrices with coefficients in $R$ and their characteristic polynomial $\det(xI-A_{ii})$ are irreducible in $K[x]$.
\end{proposition}

If $f$ is irreducible, to use the method of completing the square (this requires $\mathrm{char}(K)\neq2$), we need to discuss two cases based on whether $2\in R$ is a unit (see \S\ref{s4.2} and \S\ref{s4.3}). According to Proposition \ref{8} or Theorem \ref{2}, classifying matrices over $R$ is equivalent to classifying ideal classes in the order $\OK=R[\theta]$, where $\theta$ is a root of $f$. Let us review the definition here.

\begin{definition}\label{13}
Let $R$ be an integral domain with fractional field $K$, an order $\OK$ in a finite dimensional $K$-algebra $A$ is a subring $\OK\subseteq A$ which is a finite $R$-module with the property that $\OK\otimes_RK=A$.
\end{definition}

Recall that two ideals $J_1,J_2$ in $\OK$ are equivalent if $\alpha_1J_1=\alpha_2J_2$ for some non-zero $\alpha_1,\alpha_2\in\OK$. Therefore, one can define a set
\[\Cl(\OK):=\{\text{ideal classes in }\OK\}.\]
An important problem in number theory is determining whether $\Cl(\OK)$ is finite. If so, the cardinality of this set is called the class number of $\OK$.

When $\mathrm{char}(K)=2$, it is necessary to discuss whether $f$ is separable, since we focus on $2\times2$ matrices. The separable case can still employ the techniques developed for $\mathrm{char}(K)\neq2$ (Lemma \ref{15}); but when $f$ is inseparable, it degenerates to $f=x^2-b$, and we need to classify matrices with such characteristic polynomials (see \S\ref{s4.4}).
\subsection{$f$ is Reducible}\label{s4.1}
In this case, $f$ splits into two linear factors.

\begin{proposition}\label{14}
All the similarity classes in
\[\Big\{A\in\Mat_2(R):\text{ characteristic polynomial }f=(x-\lambda_1)(x-\lambda_2)\text{ of }A\text{ is reducible}\Big\}\]
are represented by matrices
\[\left(\begin{array}{cc}\lambda_1&\tau\\0&\lambda_2\end{array}\right),\]
where $v(\lambda_1)\geq v(\lambda_2)\geq0$ and $\tau\in\left\{\pi^0,\pi^1,\cdots,\pi^{v(\lambda_1-\lambda_2)}\right\}$.
\end{proposition}
\begin{proof}
By Proposition \ref{12}, we only need to study the case of upper triangular matrices. It is evident that
\[\left(\begin{array}{cc}\lambda_1&\tau\\0&\lambda_2\end{array}\right)\text{ and }\left(\begin{array}{cc}\lambda_2&\tau\\0&\lambda_1\end{array}\right)\]
are similar, so we can assume $v(\lambda_1)\geq v(\lambda_2)$. Now two matrices
\[\left(\begin{array}{cc}\lambda_1&\tau_1\\0&\lambda_2\end{array}\right)\text{ and }\left(\begin{array}{cc}\lambda_1&\tau_2\\0&\lambda_2\end{array}\right)\quad(\tau_1,\tau_2\neq0)\]
are similar if there exists an invertible matrix $\left(\begin{array}{cc}x&y\\z&w\end{array}\right)\in\Mat_2(R)$ such that \[\left(\begin{array}{cc}x&y\\z&w\end{array}\right)\left(\begin{array}{cc}\lambda_1&\tau_1\\0&\lambda_2\end{array}\right)=\left(\begin{array}{cc}\lambda_1&\tau_2\\0&\lambda_2\end{array}\right)\left(\begin{array}{cc}x&y\\z&w\end{array}\right).\]
This implies $z=0$ and $(\lambda_1-\lambda_2)y=\tau_1x-\tau_2w$. We discuss it in the following cases:
\begin{itemize}
\item If $v(\tau_1)=v(\tau_2)=n$, one can choose
    \[\left(\begin{array}{cc}x&y\\z&w\end{array}\right)=\left(\begin{array}{cc}\frac{\tau_2}{\pi^n}&0\\0&\frac{\tau_1}{\pi^n}\end{array}\right)\]
    as the transition matrix.

\item If $v(\tau_1)<v(\tau_2)$ and $v(\lambda_1-\lambda_2)>v(\tau_1)$, since $x$ and $w$ are invertible, one has $v(\tau_1)=v(\tau_1x-\tau_2w)=v((\lambda_1-\lambda_2)y)>v(\tau_1)$, which is a contradiction. Hence, any matrix in this case determines exactly one similarity class.

\item If $v(\tau_1)<v(\tau_2)$ and $v(\lambda_1-\lambda_2)\leq v(\tau_1)$, one can choose
    \[\left(\begin{array}{cc}x&y\\z&w\end{array}\right)=\left(\begin{array}{cc}1&\frac{\tau_1-\tau_2}{\lambda_1-\lambda_2}\\0&1\end{array}\right)\]
    as the transition matrix.
\end{itemize}
These are all the similarity classes when $f$ is reducible.
\end{proof}

From Proposition \ref{14}, we infer that there exist two matrices which are similar in $\Mat_2(K)$ and in $\Mat_2(\kappa)$ modulo $\pi$, but not similar in $\Mat_2(R)$.

\subsection{$f$ is Irreducible and $2$ is a Unit ($\mathrm{char}(K)\neq2$)}\label{s4.2}
If $f$ is irreducible, let $\theta$ be a root of $f$. In this section and \S\ref{s4.3}, we do not require $f$ to be separable. To apply Proposition \ref{8}, one must examine the structure of ideals in $\OK=R[\theta]$.

Since every ideal $J$ in $\OK$ is a rank $2$ free $R$-module, we can express $J$ as $R\pi^n\oplus R(r+s\theta)$, where $r,s\in R$ and $n\geq0$. We can safely set $s=1$, because by extracting some power of $\pi$ the only scenario where $J$ is not $(1)$ occurs when $v(s)=0$.

\begin{lemma}\label{15}
Let $r$ be some element in $R$, then all the equivalence classes in
\[\big\{R\pi^n\oplus R(r+\theta)\text{ is an ideal}:n\geq0\big\}\]
are exactly given by $0\leq n\leq\min\left\{v(2r+a),\frac{1}{2}v(b-r(r+a))\right\}$.
\end{lemma}
\begin{proof}
Write $T:=-f(-r)=b-r(r+a)$. By Proposition \ref{8}, we only need to determine whether the matrices
\[\left(\begin{array}{cc}-r&\pi^k\\\pi^{-k}T&a+r\end{array}\right)\text{ and }\left(\begin{array}{cc}-r&\pi^n\\\pi^{-n}T&a+r\end{array}\right)\]
are similar. Suppose $U:=\left(\begin{array}{cc}x&y\\z&w\end{array}\right)\in\Mat_2(R)$ satisfy
\[\left(\begin{array}{cc}x&y\\z&w\end{array}\right)\left(\begin{array}{cc}-r&\pi^k\\\pi^{-k}T&a+r\end{array}\right)=\left(\begin{array}{cc}-r&\pi^n\\\pi^{-n}T&a+r\end{array}\right)\left(\begin{array}{cc}x&y\\z&w\end{array}\right),\]
this means
\[\pi^{-k}Ty=\pi^nz,\quad\pi^kx+(2r+a)y=\pi^nw\]
and \[\det(U)=\pi^{n-k}w^2-\pi^{-(k+n)}y^2T-\pi^{-k}(2r+a)yw.\]
Without loss of generality, we assume $k<n$. First, let us suppose $v(T)=k+n$. Consider \[y=\pi^{\max\{0,k-v(2r+a)\}}.\]
If $v(y)=0$, we can choose $w=\pi^{k+1}$; if $v(y)>0$, we can choose $w=1$. Consequently, $U$ can be invertible, allowing us to assume $k<n\leq\frac{1}{2}v(T)$. In this scenario, $v(T)>k+n$, and $U$ can be invertible if and only if $v(2r+a)\leq k$. Thus, when $k<v(2r+a)$, these two ideals are not equivalent.
\end{proof}

This lemma does not require the invertibility of $2\in R$, even the characteristic can be $2$. With this lemma in hand, when $2\in R$ is a unit, we can classify the ideal classes of $\OK$.

\begin{lemma}\label{16}
If $2\in R$ is a unit, then any ideal class in $\OK=R[\theta]$ can be represented by an ideal of the form
\[R\pi^k\oplus R\theta',\]
where $k\geq0$ is an integer and $\theta'\in\OK\setminus R$ is a generator in some integral base. More precisely, if $\theta$ is a root of $f(x)=x^2-ax-b$, then any ideal class can be represented by $R\pi^k\oplus R\left(\theta-\frac{a}{2}\right)$.
\end{lemma}
\begin{proof}
Let us consider the case where $a=0$. For an ideal $J=R\pi^n\oplus R(r+\theta)$, we claim that $r=0$. If not, we proceed with the following assumptions:
\begin{itemize}
\item[(1)] $v(r)<n$. Otherwise, one can utilize $\pi^n$ to decrease the valuation of $r$.
\item[(2)] $v(b-r^2)>v(r)+n$. This is because $(\theta-r)J=R(b-r^2)\oplus R(r\pi^n-\pi^n\theta)\subseteq J$, allowing us to utilize $T:=b-r^2$ to decrease $v(r\pi^n)$.
\end{itemize}
However, Lemma \ref{15} implies $0\leq n\leq\min\{v(r),\frac{1}{2}v(T)\}$, which contradicts (1).

If $a\neq0$, it is easy to see $\theta':=\theta-\frac{a}{2}$ is a root of the irreducible polynomial $x^2-\left(\frac{a^2}{4}+b\right)\in R[x]$. Consequently, we can follow the aforementioned steps for $R[\theta']=R[\theta]$. Thus, any ideal class in $R[\theta]$ must take the form $R\pi^n\oplus R\theta'$.
\end{proof}

Using this classification, we can list the representative element from each similarity class, which will play the role of standard forms of similarity.

\begin{proposition}\label{17}
If $2\in R$ is a unit, then all similarity classes in
\[\Big\{A\in\Mat_2(R):\text{ characteristic polynomial }f=x^2-ax-b\text{ of }A\text{ is irreducible}\Big\}\]
are represented by matrices
\[\left(\begin{array}{cc}\frac{a}{2}&\pi^k\\\left(\frac{a^2}{4}+b\right)\pi^{-k}&\frac{a}{2}\end{array}\right),\]
where $v\left(\frac{a^2}{4}+b\right)\geq2k\geq0$.
\end{proposition}
\begin{proof}
By Proposition \ref{8}, the problem reduces to classifying the ideal classes in $\OK$, and this is provided by Lemma \ref{16}. It is straightforward to confirm that any ideal taking the form $R\pi^k\oplus R\left(\theta-\frac{a}{2}\right)$ corresponds to the matrix stated above. Furthermore, if there exists a matrix $\left(\begin{array}{cc}x&y\\z&w\end{array}\right)\in\Mat_2(R)$ such that (assume $k<l$)
\[\left(\begin{array}{cc}x&y\\z&w\end{array}\right)\left(\begin{array}{cc}\frac{a}{2}&\pi^k\\\left(\frac{a^2}{4}+b\right)\pi^{-k}&\frac{a}{2}\end{array}\right)=\left(\begin{array}{cc}\frac{a}{2}&\pi^l\\\left(\frac{a^2}{4}+b\right)\pi^{-l}&\frac{a}{2}\end{array}\right)\left(\begin{array}{cc}x&y\\z&w\end{array}\right),\]
then $w\pi^l=x\pi^k$ and $z=\left(\frac{a^2}{4}+b\right)y\pi^{-k-l}$. If $k=v\left(\frac{a^2}{4}+b\right)-l$, one can select
\[\left(\begin{array}{cc}x&y\\z&w\end{array}\right)=\left(\begin{array}{cc}\pi^{l-k}&1\\\frac{\frac{a^2}{4}+b}{\pi^{k+l}}&1\end{array}\right),\]
which is invertible. Therefore, we may assume $k<l\leq\frac{1}{2}v\left(\frac{a^2}{4}+b\right)$. Now, the determinant is not in $R^{\times}$ since $v(x)=v(\pi^{l-k} w)>0$ and $v(z)=v\left(\left(\frac{a^2}{4}+b\right)y\pi^{-k-l}\right)>0$, indicating that these two matrices cannot be similar.
\end{proof}

Under the assumptions above, as a direct corollary, the class number of $\OK$ is
\[\#\Cl(\OK)=\left\lfloor\frac{1}{2}v\left(\frac{a^2}{4}+b\right)\right\rfloor+1.\]
\subsection{$f$ is Irreducible and $2$ is Not a Unit ($\mathrm{char}(K)\neq2$)}\label{s4.3}
We continue to denote $f(x)=x^2-ax-b$, and $\OK=R[\theta]$ is the integral extension determined by this polynomial. Let us assume the valuation of $2\in R$ is $v(2)=e>0$. We begin with the following lemma. Recall that by Lemma \ref{15}, any ideal in $\OK$ takes the form $R\pi^n\oplus R(r+\theta)$.

\begin{lemma}\label{18}
Suppose $R$ has a uniformizer $\pi$ and $v(2)=e>0$.
\begin{itemize}
\item[(1)] If $v(a)<e$, let $m$ be the maximum integer in $\{0,1,\cdots,v(a)\}$ such that
    \[v(b-r(r+a))\geq2m\]
    for some $r\in R$. Then, all the ideal classes in $\OK$ are
    \[R\pi^k\oplus R(r+\theta),\quad k=0,1,\cdots,m.\]
\item[(2)] For $v(a)\geq e$, write $\Delta:=\frac{a^2}{4}+b$, $\theta=\frac{a}{2}+\sqrt{\Delta}$.
\begin{itemize}
\item[(2.1)] If $v(\Delta)$ is odd, then all the ideal classes in $\OK$ are
    \[R\pi^k\oplus R\left(\theta-\frac{a}{2}\right),\quad k=0,1,\cdots,\left\lfloor\frac{1}{2}v(\Delta)\right\rfloor.\]
\item[(2.2)] If $v(\Delta)$ is even, let $m$ be the maximum integer in $\{0,1,\cdots,e\}$ such that
    \[v(\Delta-r^2)\geq2m+v(\Delta)\]
    for some $r\in R$, then all the ideal classes in $\OK$ are
    \[R\pi^k\oplus R\left(\theta-\frac{a}{2}\right),\quad k=0,1,\cdots,\frac{1}{2}v(\Delta)\]
    and
    \[R\pi^{\frac{1}{2}v(\Delta)+i}\oplus R\left(r+\theta-\frac{a}{2}\right),\quad i=1,\cdots,m~(\text{if }m\geq1).\]
\end{itemize}
\end{itemize}
Moreover, the class number of $\OK$ is
\[\#\Cl(\OK)=\left\{\begin{array}{ll}m+1,&\text{ case (1)}\\ \left\lfloor\frac{1}{2}v(\Delta)\right\rfloor+1,&\text{ case (2.1)}\\ \frac{1}{2}v(\Delta)+m+1,&\text{ case (2.2)}\end{array}\right..\]
\end{lemma}
\begin{proof}
In case (1), we have $v(2r+a)=v(a)$. Hence, for an integer $m$ and some $r\in R$ given by the condition, Lemma \ref{15} yields ideal classes $R\pi^k\oplus R(r+\theta)$ with $k=0,1,\cdots,m$. If there exists another $r'$ satisfying the condition, we assert $R\pi^k\oplus R(r+\theta)$ and $R\pi^k\oplus R(r'+\theta)$ are equivalent. Indeed, one can show $v(r-r')\geq k$. If not, suppose $v(r_0)<k$ with $r_0=r-r'$, then
\begin{align*}
2k\leq&v(b-r(r+a))\\=&v(b-(r_0+r')(r_0+r'+a))\\=&v(b-r'(r'+a)-r_0^2-r_0(a+2r'))\\=&v(r_0^2)\\<&2k,
\end{align*}
since $v(b-r'(r'+a))\geq2k$, $v(r_0^2)<2k$ and $v(r_0)<k\leq v(a)=v(a+2r')$. This is a contradiction.

In case (2), one can set $\theta':=\theta-\frac{a}{2}$ to simplify (2) to the case of $f(x)=x^2-\Delta$. Here $\theta'=\sqrt{\Delta}$.
\begin{itemize}
\item We first consider the case $v(\Delta)=0$. If $v(\Delta-r^2)=0$, there is only one ideal class $\OK$. So we may assume $v(\Delta-r^2)>0$, therefore $r\in R^{\times}$ and $v(2r)=e$. By Lemma \ref{15} and the definition of $m$, there exist finitely many distinct ideal classes
    \[R\pi^i\oplus R(r+\sqrt{\Delta})\]
    with $i=0,1,\cdots,m$. Utilizing the same argument as in the proof of case (1), one can show that altering $r$ to another one such that $v(2r)=e$ does not alter the ideal class.
\item Suppose $v(\Delta)>0$. If $r=0$ we obtain the ideal classes of the form
    \[R\pi^k\oplus R\sqrt{\Delta},\]
    where $k=0,1,\cdots,\left\lfloor\frac{1}{2}v(\Delta)\right\rfloor$. Now, let us consider the case where $r\neq0$. According to Lemma \ref{15}, the only condition for additional ideals of the form
    \[R\pi^i\oplus R(r+\sqrt{\Delta})\]
    to emerge is $\frac{1}{2}v(\Delta-r^2)>v(r)$. This implies that $v(\Delta)=v(r^2)$ is an even number, thus $2m+v(\Delta)\leq v(\Delta-r^2)$ for some $m\geq1$. Suppose $m\leq e$ is the largest number satisfying $2m+v(\Delta)\leq v(\Delta-r^2)$, then the additional ideals precisely take the form
    \[R\pi^{\frac{1}{2}v(\Delta)+i}\oplus R(r+\sqrt{\Delta}),\]
    where $i=1,\cdots,m$. Here, we require $m\leq e$ because $\frac{1}{2}v(\Delta)+i\leq v(r)+m\leq v(2r)=v(r)+e$. The equivalence classes also independent of $r$ when $v(r)=\frac{1}{2}v(\Delta)$. Additionally, these ideals are not equivalent to the ideals in the case $r=0$, since if we express $\Delta=\Delta_0\pi^{2v(r)}$ and $r=r_0\pi^{v(r)}$ for some $\Delta_0,r_0\in R^{\times}$, under the ring homomorphism $R[\sqrt{\Delta}]\hookrightarrow R[\sqrt{\Delta_0}]$ they are mapped to the classes $\OK$ and $R\pi^i\oplus R(r_0+\sqrt{\Delta_0})$, respectively.
\end{itemize}
Hence, the conclusion is validated.
\end{proof}

As before, we can list the standard representative elements from each similarity class as follows.

\begin{proposition}\label{19}
Under the assumptions of Lemma \ref{18}, all the similarity classes in
\[\Big\{A\in\Mat_2(R):\text{ characteristic polynomial }f=x^2-ax-b\text{ of }A\text{ is irreducible}\Big\}\]
are represented by matrices:
\begin{itemize}
\item Case (1): $\left(\begin{array}{cc}-r&\pi^i\\\frac{b-r(r+a)}{\pi^i}&a+r\end{array}\right)$, where $i=0,1,\cdots,m$.

\item Case (2.1): $\left(\begin{array}{cc}\frac{a}{2}&\pi^n\\\Delta\pi^{-n}&\frac{a}{2}\end{array}\right)$, where $n=0,1,\cdots,\left\lfloor\frac{v(\Delta)}{2}\right\rfloor$.

\item Case (2.2): $\left(\begin{array}{cc}\frac{a}{2}&\pi^n\\\Delta\pi^{-n}&\frac{a}{2}\end{array}\right)$, where $n=0,1,\cdots,\frac{v(\Delta)}{2}$; and $\left(\begin{array}{cc}\frac{a}{2}-r&\pi^{v(\Delta)/2+i}\\ \frac{\Delta-r^2}{\pi^{v(\Delta)/2+i}}&\frac{a}{2}+r\end{array}\right)$,
    where $i=1,\cdots,m$ (if $m\geq1$).
\end{itemize}
Here we denote $\Delta:=\frac{a^2}{4}+b$.
\end{proposition}
\begin{proof}
This is Proposition \ref{8} and Lemma \ref{18}.
\end{proof}

We compare our result with the corresponding classification that we are familiar with in finite fields.

Fix a prime $p\neq2$, let $F$ be a finite field of characteristic $p$, then any element $g\in\GL_2(F)$ with characteristic polynomial $f\in F[x]$ must fall into one of the following conjugacy classes:
\begin{itemize}
\item $\left(\begin{array}{cc}\alpha&0\\0&\beta\end{array}\right)$, if $f$ has two distinct roots $\alpha,\beta\in F^{\times}$.

\item $\left(\begin{array}{cc}\alpha&0\\0&\alpha\end{array}\right)$ or $\left(\begin{array}{cc}\alpha&1\\0&\alpha\end{array}\right)$, if $f$ has a multiple root $\alpha\in F^{\times}$.
\end{itemize}
The above two classes correspond to the case in \S\ref{s4.1}.
\begin{itemize}
\item $\left(\begin{array}{cc}0&b\\1&a\end{array}\right)$, if $f=x^2-ax-b$ is irreducible in $F[x]$.
\end{itemize}
This class comes from the degenerate case of Proposition \ref{19}.
\subsection{$f$ is Irreducible and $\mathrm{char}(K)=2$}\label{s4.4}
Let $f(x)=x^2-ax-b\in R[x]$ be an irreducible polynomial and $\OK=R[\theta]$ be the integral extension of $R$ by $f$. First, assume $f$ is separable, then $a\neq0$. Since Lemma \ref{15} still holds, note that $v(2)=v(0)=\infty$, the case (1) in Lemma \ref{18} becomes:

\begin{lemma}\label{20}
Suppose $R$ has a uniformizer $\pi$ and $f$ is separable. Let $m$ be the maximum integer in $\{0,1,\cdots,v(a)\}$ such that
\[v(b-r(r+a))\geq2m\]
for some $r\in R$. Then, all the ideal classes in $\OK$ are
\[R\pi^k\oplus R(r+\theta),\quad k=0,1,\cdots,m.\]
Hence, all the similarity classes in
\[\Big\{A\in\Mat_2(R):\text{ characteristic polynomial }f=x^2-ax-b\text{ of }A\text{ is separable}\Big\}\]
are represented by $\left(\begin{array}{cc}-r&\pi^i\\\frac{b-r(r+a)}{\pi^i}&a+r\end{array}\right)$, $i=0,1,\cdots,m$.
\end{lemma}

However, if $f$ is inseparable, our discussion will become more complicated. In this case $a=0$, so $f(x)=x^2-b$, where $b\in R\setminus R^2$. Since the characteristic is $2$, any matrix with this characteristic polynomial must have the form $\left(\begin{array}{cc}u&s\\t&u\end{array}\right)\in\Mat_2(R)$ with $u^2+st=b$.

Suppose we have two matrices $\left(\begin{array}{cc}u&s\\t&u\end{array}\right)$ and $\left(\begin{array}{cc}u'&s'\\t'&u'\end{array}\right)$ with the same characteristic polynomial $f(x)=x^2-b$. Because the matrix here is similar to its transpose, without loss of generality we may assume that $v(t)\leq v(s)$, $v(t')\leq v(s')$ and $v(t')\leq v(t)$.

\begin{lemma}\label{21}
Under the assumptions above,
\[\left(\begin{array}{cc}u&s\\t&u\end{array}\right)\text{ and }\left(\begin{array}{cc}u'&s'\\t'&u'\end{array}\right)\]
are similar if and only if $v(t)=v(t')$. Hence, all the similarity classes in
\[\Big\{A\in\Mat_2(R):\text{ characteristic polynomial }f=x^2-b\text{ of }A\text{ is inseparable}\Big\}\]
are represented by $\left(\begin{array}{cc}u_i&s_i\\\pi^i&u_i\end{array}\right)$, $i\geq0$, where we choose one solution $u_i,s_i$ of $u_i^2+s_i\pi^i=b,v(s_i)\geq i$ for each $i$.
\end{lemma}
\begin{proof}
Let $U\left(\begin{array}{cc}u&s\\t&u\end{array}\right)=\left(\begin{array}{cc}u'&s'\\t'&u'\end{array}\right)U$ with $U=\left(\begin{array}{cc}x&y\\z&w\end{array}\right)\in\GL_2(R)$, this is equivalent to the following linear equations system
\[\left(\begin{array}{cccc}u+u'&t&s'&0\\s&u+u'&0&s'\\t'&0&u+u'&t\\0&t'&s&u+u'\end{array}\right)\left(\begin{array}{c}x\\y\\z\\w\end{array}\right)=0.\]
Since $u^2+st=u'^2+s't'$, which means $(u+u')^2=st+s't'$, we have $2v(u+u')=v(st+s't')\geq2v(t')$, so $v(u+u')\geq v(t')$. According to the assumptions $v(t')\leq v(t)\leq v(s)$ and $v(t')\leq v(s')$, the Gaussian elimination method gives
\[\left(\begin{array}{cccc}t'&0&u+u'&t\\0&t'&s&u+u'\\0&0&(u+u')^2+st+s't'&0\\0&0&0&(u+u')^2+st+s't'\end{array}\right)\left(\begin{array}{c}x\\y\\z\\w\end{array}\right)=0.\]
Note that $(u+u')^2+st+s't'=(u^2+st)+(u'^2+s't')=2b=0$, so we can directly obtain the solution of the above linear equation
\[x=\frac{(u+u')z+tw}{t'}\quad\text{and}\quad y=\frac{sz+(u+u')w}{t'}.\]
Obviously, $\det(U)=xw-yz=\frac{tw^2+sz^2}{t'}\in R^{\times}$ if and only if $v(t)=v(t')$, in which case $U=\left(\begin{array}{cc}t/t'&(u+u')/t'\\0&1\end{array}\right)$.
\end{proof}

We emphasize that the class number in this inseparable case is infinite.

\subsection{A Descent Problem}\label{s4.5}
To summarize, combined with Lemma \ref{16}, Proposition \ref{17}, Lemma \ref{18}, Lemma \ref{20} and Lemma \ref{21}, we have

\begin{theorem}\label{22}
Let $(R,\pi)$ be a DVR with a valuation $v$. Suppose $f(x)=x^2-ax-b\in R[x]$ is a monic irreducible polynomial, and $\theta$ be a root of $f$. Consider $\OK=R[\theta]$, the integral extension. Then, the class number of $\OK$ is finite, unless $f$ is inseparable and $\mathrm{char}(K)=2$. To be more precise,
\begin{itemize}
\item If $2\in R$ is a unit, then any ideal class in $\OK$ can be represented by an ideal of the form
    \[R\pi^k\oplus R\left(\theta-\frac{a}{2}\right),\quad0\leq k\leq\frac{1}{2}v\left(\frac{a^2}{4}+b\right).\]

\item Suppose $2\in R$ has valuation $0<e<\infty$.
\begin{itemize}
\item If $v(a)<e$, let $m$ be the maximum integer in $\{0,1,\cdots,v(a)\}$ such that $v(b-r(r+a))\geq2m$ for some $r\in R$. Then, all the ideal classes in $\OK$ are
    \[R\pi^k\oplus R(r+\theta),\quad0\leq k\leq m.\]

\item For $v(a)\geq e$, write $\Delta:=\frac{a^2}{4}+b$, $\theta=\frac{a}{2}+\sqrt{\Delta}$.
\begin{itemize}
\item If $v(\Delta)$ is odd, then all the ideal classes in $\OK$ are
    \[R\pi^k\oplus R\left(\theta-\frac{a}{2}\right),\quad0\leq k\leq\left\lfloor\frac{1}{2}v(\Delta)\right\rfloor.\]
\item If $v(\Delta)$ is even, let $m$ be the maximum integer in $\{0,1,\cdots,e\}$ such that $v(\Delta-r^2)\geq2m+v(\Delta)$ for some $r\in R$, then all the ideal classes in $\OK$ are
    \[R\pi^k\oplus R\left(\theta-\frac{a}{2}\right),\quad0\leq k\leq\frac{1}{2}v(\Delta)\]
    and
    \[R\pi^{\frac{1}{2}v(\Delta)+i}\oplus R\left(r+\theta-\frac{a}{2}\right),\quad1\leq i\leq m~(\text{if }m\geq1).\]
\end{itemize}
\end{itemize}
\item Otherwise, suppose $0=2\in R$.
\begin{itemize}
\item If $f$ is separable, let $m$ be the maximum integer in $\{0,1,\cdots,v(a)\}$ such that $v(b-r(r+a))\geq2m$ for some $r\in R$. Then, all the ideal classes in $\OK$ are
    \[R\pi^k\oplus R(r+\theta),\quad0\leq k\leq m.\]

\item If $f$ is inseparable (in this case $a=0$ and $b=\theta^2$), all the ideal classes in $\OK$ are
\[Rs_i\oplus R(u_i+\theta),\quad i\geq0,\]
where we fix a pair $u_i,s_i$ satisfies the equation $u_i^2+s_i\pi^i=b,v(s_i)\geq i$ for each $i$.
\end{itemize}
\end{itemize}
\end{theorem}

As an application of Theorem \ref{22}, let us now prove the remaining case of Theorem \ref{3}, which is a special case of Conjecture \ref{1} when $n=2$. Although some parts of the proof have already been addressed in \S\ref{s2}, we still use our classification to reprove it here. Let us restate the conclusion in more detail as follows:

\begin{theorem}\label{23}
Suppose $R$ is a DVR with fractional field $K$. Let $L/K$ be a finite field extension, and $S$ be the integral closure of $R$ in $L$. Then, two matrices $A,B\in\Mat_2(R)$ are similar over $R$ if and only if they are similar over $S$.
\end{theorem}

Regarding the proof, our task solely involves establishing the 'if' part, a task achievable through our classification of similarity classes. Notably, the only non-trivial case is when the characteristic polynomial of a matrix, denoted as $f(x)=x^2-ax-b$, is irreducible in $R[x]$ but reducible in $S[x]$, i.e. $f(x)=(x-\lambda_1)(x-\lambda_2)$ for some $\lambda_1,\lambda_2\in S$. We only discuss the proof of this case for $2\in R$ is not a unit. This is further decomposed into two smaller cases: $0<v(2)<\infty$ and $v(2)=\infty$ (the $\mathrm{char}(K)=2$ case).

\begin{proof}
\begin{itemize}
\item (the $0<v(2)<\infty$ case, see Theorem \ref{6}). Let us suppose $S$ over $R$ has ramification index $e$ with uniformizer $\Pi\in S$ and assume $\pi=\Pi^e\in R$. Let $v$ and $v'$ denote valuations on $K^{\times}$ and $L^{\times}$ respectively, where $v'=e\cdot v$. Without loss of generality, we assume $m=v(a)$ in case (1) and $m=v(2)$ in case (2.2) as outlined in Proposition \ref{19}.

If $v(a)<v(2)$, we conclude that
\[v'((\lambda_1-\lambda_2)^2)=v'((\lambda_1+\lambda_2)^2-4\lambda_1\lambda_2)=v'(a^2),\]
since $v'(4b)=2v'(2)+v'(b)>v'(a^2)$; and
\[v'(\lambda_1+r)\geq v'(a),\]
since $v'(b-r(r+a))=v'(\lambda_1+r)+v'(\lambda_1+r-(a+2r))\geq2v'(a)$. Similarly, $v'(\lambda_2+r)\geq v'(a)$. Suppose
\[\left(\begin{array}{cc}-r&\Pi^{ei}\\\frac{b-r(r+a)}{\Pi^{ei}}&a+r\end{array}\right)\left(\begin{array}{cc}x&y\\z&w\end{array}\right)=\left(\begin{array}{cc}x&y\\z&w\end{array}\right)\left(\begin{array}{cc}\lambda_1&\Pi^t\\0&\lambda_2\end{array}\right)\]
for some $U:=\left(\begin{array}{cc}x&y\\z&w\end{array}\right)\in\GL_2(S)$, $i=0,1,\cdots,v(a)\leq v'(a)$. From the above calculation, when $t=\infty$ one can choose
\[U=\left(\begin{array}{cc}1&-1\\\frac{\lambda_1+r}{\Pi^{ei}}&-\frac{\lambda_2+r}{\Pi^{ei}}\end{array}\right),\]
in this case $\det(U)=(\lambda_1-\lambda_2)\Pi^{-ei}\in S^{\times}$ if and only if $i=v(a)$. When $t<\infty$, by Proposition \ref{14} we have $t<v'(a)$, so if one chooses
\[U=\left(\begin{array}{cc}1&-1\\\frac{\lambda_1+r}{\Pi^{ei}}&\frac{\Pi^t-(\lambda_2+r)}{\Pi^{ei}}\end{array}\right),\]
then $v'(\det(U))=v'((\Pi^t+(\lambda_1-\lambda_2))\Pi^{-ei})=t-ei$. This gives the similarity classes when $0\leq i<v(a)$.

If $v(a)\geq v(2)$, suppose $f(x)=x^2-\Delta\in R[x]$. We only consider the case (2.2) and the case (2.1) is easier. Since $v'(\sqrt{\Delta})=\frac{e}{2}v(\Delta)$, we have
\begin{multline*}\left(\begin{array}{cc}0&\Pi^{en}\\\Delta\Pi^{-en}&0\end{array}\right)\left(\begin{array}{cc}1&-1\\\Pi^{-en}\sqrt{\Delta}&\Pi^{-en}(\Pi^t+\sqrt{\Delta})\end{array}\right)\\=\left(\begin{array}{cc}1&-1\\\Pi^{-en}\sqrt{\Delta}&\Pi^{-en}(\Pi^t+\sqrt{\Delta})\end{array}\right)\left(\begin{array}{cc}\lambda_1&\Pi^t\\0&\lambda_2\end{array}\right),\end{multline*}
where $n=0,1,\cdots,\frac{1}{2}v(\Delta)$ such that the transition matrices are in $\GL_2(S)$, implying the corresponding $t$ are $0,e,\cdots,\frac{e}{2}v(\Delta)$. For additional matrices in case (2.2) of Proposition \ref{19}, the same calculations show that $t=e\left(\frac{v(\Delta)}{2}+i\right)$ when $i<v(2)$, and $t=\infty$ when $i=v(2)$.

\item (the $\mathrm{char}(K)=2$ case). Here we need to discuss whether $f$ is separable. However, when $f$ is separable, our proof is similar to the $v(2)<\infty$ case as above. So we only present the argument that $f$ is inseparable. Indeed, it only needs to be noticed that
    \[\left(\begin{array}{cc}0&1\\1&\frac{\theta+u_i}{\pi^i}\end{array}\right)\left(\begin{array}{cc}u_i&s_i\\\pi^i&u_i\end{array}\right)=\left(\begin{array}{cc}\theta&\pi^i\\0&\theta\end{array}\right)\left(\begin{array}{cc}0&1\\1&\frac{\theta+u_i}{\pi^i}\end{array}\right),\quad i\geq0\]
    by Proposition \ref{14} and Lemma \ref{21}, thus the similarity classes with characteristic polynomial $f(x)=x^2-\theta^2\in R[x]$ correspond one-to-one with the similarity classes with characteristic polynomial $f(x)=(x-\theta)^2\in S[x]$.
\end{itemize}
This completes the proof of our theorem.
\end{proof}

In essence, Theorem \ref{23} implies the existence of an injective map
\[\Cl(R[\theta])\hookrightarrow\Cl(S[\theta])\]
if $f$ is always irreducible.

Here we provide an example to elucidate our theorems.

\begin{example}\label{24}
Consider $R=\Z_2$ and $\OK=\Z_2[\sqrt{5}]$ with $f(x)=x^2-5$. By Theorem \ref{22} we have
\[\Cl(\OK)=\{\OK,R2\oplus R(1+\sqrt{5})\}.\]
The principal ideal $\OK$ corresponds to the similarity class
\[A:=\left(\begin{array}{cc}0&1\\5&0\end{array}\right)\in\Mat_2(R)\]
and the non-trivial ideal $R2\oplus R(1+\sqrt{5})$ corresponds to the similarity class
\[B:=\left(\begin{array}{cc}-1&2\\2&1\end{array}\right)\in\Mat_2(R).\]
Normalizing $\OK$ we obtain $S:=\Z_2[\frac{1+\sqrt{5}}{2}]$. By Theorem \ref{23}, we deduce that $A,B\in\Mat_2(S)$ are not similar. Indeed, $f(x)=(x+\sqrt{5})(x-\sqrt{5})$ splits in $S$, so all the similarity classes in $\Mat_2(S)$ must have the form
\[C:=\left(\begin{array}{cc}\sqrt{5}&0\\0&-\sqrt{5}\end{array}\right)\text{ and }D:=\left(\begin{array}{cc}\sqrt{5}&1\\0&-\sqrt{5}\end{array}\right)\]
by Proposition \ref{14}. One can check that $A$ is similar to $D$, and $B$ is similar to $C$ as matrices in $\Mat_2(S)$.
\end{example}

\section{Over Dedekind Rings of Algebraic Integers}\label{s5}
In this section, we work on the ring of integers of a number field.

\subsection{Changing Basis}\label{s5.1}

Let $K$ be a number field with integers $R$, this is a Dedekind domain.

\begin{definition}\label{25}
Let $V$ be a $n$ dimensional $K$-linear space. we call an $R$-module $J\subseteq V$ a lattice in $V$ with respect to $R$ if there is a $V$-basis such that
\[J\subseteq Rx_1+\cdots+Rx_n.\]
We say that $J$ is a full lattice if, in addition to the above property, we have $J\otimes_RK=V$.
\end{definition}

For instance, an order and its ideals are lattices.

Consider the lattice $J$ in $V$. For any non-zero vector $x$ in $J\otimes_RK$, we define the coefficient of $x$ with respect to $J$ to be the fractional ideal
\[\{k\in K:kx\in J\}.\]
The following Proposition is \cite[81.2]{om}:

\begin{proposition}\label{26}
Given a full lattice $J$ in $V$, a hyperplane $U$ in $V$, and a vector $x_0\in V\setminus U$. Then among all vectors in $x_0+U$ there is at least one whose coefficient with respect to $J$ is largest. Indeed, this coefficient denotes as $\mathfrak{a}$ is
\[\mathfrak{a}=\{k\in K:kx_0\in J+U\}.\]
Then, for any vector $x_0+u_0$ ($u_0\in U$) with coefficient $\mathfrak{a}$ we have
\[J=\mathfrak{a}(x_0+u_0)\oplus(J\cap U).\]
\end{proposition}

It is not necessary to explicitly determine the $u_0$ in Proposition \ref{26} in order to assess whether an ideal is a free $R$-module. Indeed, the problems can all be reduced to the case of $[V:K]=2$ by induction, since $J\cap U$ is also a lattice. So we only need the following:

\begin{corollary}\label{27}
Let $K$ be a number field with ring of integers $R$. Suppose $L/K$ is a quadratic field extension, then a full lattice $J$ in $L$ is free as an $R$-module if and only if
\[\mathfrak{a}\cdot(J\cap K)\]
is a principal fractional ideal in $R$, where $\mathfrak{a}=\{k\in K:kx_0\in J+K\}$ for some $x_0\in L\setminus K$.
\end{corollary}
\begin{proof}
Let $V=L$ and $U=K$ in Proposition \ref{26}. By the structure theorem for modules over a Dedekind domain \cite[Proposition III.4.3]{neu}, we have
\[J\cong R\oplus(\mathfrak{a}\cdot(J\cap K)).\]
Hence, $J$ is free as an $R$-module if and only if $\mathfrak{a}\cdot(J\cap K)$ is principal.
\end{proof}

As a special case, in the context of Corollary \ref{27}, suppose $f\in R[x]$ is a quadratic monic irreducible polynomial, let $\theta$ be a root of $f$ and $L=K[\theta]$. Then, an ideal class $J$ in $\Cl(R[\theta])$ is free as an $R$-module (hence $J$ corresponds to a similarity class in $\Mat_2(R)$ by Theorem \ref{2}) if and only if $\mathfrak{a}\cdot(J\cap K)$ is principal in $R$.

In fact, the proposed algorithm in Proposition \ref{26} is suitable for general Dedekind domains, but requires complex iterations, which are not inherently difficult except for computation. We will show how to use this conclusion with a few examples.

\subsection{Examples}\label{s5.2}
Here we provide two examples, these two examples demonstrate that there can be both non-trivial free ideals and non-free ideals within class groups. One of the examples comes from the case where $R=\Z[\sqrt{-5}]$ in \cite[23]{ao}.

\begin{example}\label{28}
Let $R=\Z[\sqrt{-5}]$ with fractional field $K=\Q[\sqrt{-5}]$, consider the field extension $L:=K[\sqrt{2}]$ with minimal polynomial $f(x)=x^2-2$. The class group of $R[\sqrt{2}]$ is $\Z/2\Z$, with a non-trivial generator
\begin{align*}
J=&R[\sqrt{2}]2+R[\sqrt{2}]\left(1+\frac{\sqrt{-5}-1}{2}\sqrt{2}\right)\\=&R2+R2\sqrt{2}+R\left(1+\frac{\sqrt{-5}-1}{2}\sqrt{2}\right)+R(\sqrt{-5}+\sqrt{2}-1).
\end{align*}
In order to determine whether this ideal is a free $R$-module, we intend to apply Corollary \ref{27}. First, a direct computation shows that
\[J\cap K=R2+R(1+\sqrt{-5}),\]
which is the non-principal element in $\Cl(R)$. To study $J+K$, by the method of undetermined coefficients, the range of values for the coefficients of $\sqrt{2}$ and $\sqrt{-10}$ in $J+K$ are listed in the table below:
\begin{table}[!ht]
    \centering
    \begin{tabular}{c|c|c|c}
    \hline
        coefficient of $\sqrt{2}$&$0$&$\frac{1}{2}$&$\cdots$\\ \hline
        coefficient of $\sqrt{-10}$&$\Z$&$\frac{1}{2}+\Z$&$\cdots$\\
        \hline
    \end{tabular}
\end{table}

So one can choose $x_0=\frac{\sqrt{2}}{2}$. At this point, it is easy to compute
\[\mathfrak{a}=R2+R(1+\sqrt{-5}),\]
and therefore, there is an $R$-module isomorphism
\[J\cong R\oplus R.\]
By Theorem \ref{2} and Corollary \ref{27}, the ideal $\mathfrak{a}\cdot(J\cap K)=R2$ is principal concludes that there are exactly two similarity classes in $\Mat_2(\Z[\sqrt{-5}])$ with characteristic polynomial $x^2-2$.
\end{example}

\begin{example}\label{29}
Let $R=\Z[\sqrt{-5}]$ with fractional field $K=\Q[\sqrt{-5}]$, let $L/K$ be a quadratic field extension with minimal polynomial $f(x)=x^2-x+7$. Suppose $\theta$ is a root of $f$, the class group of $R[\theta]$ is $\Z/2\Z$, which is generated by
\begin{align*}
J=&R[\theta]3+R[\theta]\left(1+\frac{-\theta+2}{3}\sqrt{-5}\right)\\=&R3+R3\theta+R\left(1+\frac{-\theta+2}{3}\sqrt{-5}\right)+R\left(\theta+\frac{\theta+7}{3}\sqrt{-5}\right).
\end{align*}
First, one can compute
\[J\cap K=R3+R(2+\sqrt{-5}),\]
which is the non-principal element in $\Cl(R)$. To study $J+K$, by the same method, the range of values for the coefficients of $\theta$ and $\theta\sqrt{-5}$ in $J+K$ are
\begin{table}[!ht]
    \centering
    \begin{tabular}{c|c|c|c|c}
    \hline
        coefficient of $\theta$&$0$&$\frac{1}{3}$&$\frac{2}{3}$&$\cdots$\\ \hline
        coefficient of $\theta\sqrt{-5}$&$\frac{1}{3}\Z$&$\frac{1}{3}\Z$&$\frac{1}{3}\Z$&$\cdots$\\
        \hline
    \end{tabular}
\end{table}

So one can choose $x_0=\frac{\theta}{3}$. Now, it is easy to compute
\[\mathfrak{a}=R\text{ and }J\cong R\oplus\big(R3+R(2+\sqrt{-5})\big).\]
Thus, $\mathfrak{a}\cdot(J\cap K)$ is not free as an $R$-module, which means there is only one similarity class in $\Mat_2(\Z[\sqrt{-5}])$ with characteristic polynomial $x^2-x+7$. This more general result is not surprising compared to the example in \cite[23]{ao}.
\end{example}

\end{document}